\documentclass[11pt]{amsart}
\usepackage{CJK}
\usepackage{amsmath,amssymb,bm,amsthm, amsfonts, mathrsfs, eucal, epsfig}
\usepackage{graphicx}
\usepackage{url}

\begin{document}

\newtheorem{Def}{Definition}[section]
\newtheorem{Lem}{Lemma}[section]
\newtheorem{Rem}{Remark}[section]
\newtheorem{Thm}{Theorem}[section]
\newtheorem{Cor}{Corollary}[Thm]
\newtheorem{sublemma}{Sub-Lemma}
\newtheorem{Prop}{Proposition}[section]
\newtheorem{Example}{Example}[section]
\newcommand{\g}[0]{\textmd{g}}
\newcommand{\pr}[0]{\partial_r}
\newcommand{\dif}{\mathrm{d}}
\newcommand{\bg}{\bar{\gamma}}
\newcommand{\md}{\rm{md}}
\newcommand{\cn}{\rm{cn}}
\newcommand{\sn}{\rm{sn}}
\newcommand{\seg}{\mathrm{seg}}

\title{A Remark on Regular Points of Ricci Limit Spaces}
\author{Lina Chen}
\date{\today}
\maketitle

\begin{abstract}
   Let $Y$ be a Gromov-Hausdorff limit of complete Riemannian n-manifolds with Ricci curvature bounded from below.
   A point in $Y$ is called $k$-regular, if its tangent is unique and is isometric to an $k$-dimensional Euclidean space.
   By \cite{B5}, there is $k>0$ such that the set of all $k$-regular point $\mathcal{R}_k$ has a full renormalized measure.
   An open problem is if $\mathcal{R}_l=\emptyset$ for all $l<k$? The main result in this paper asserts that
   if $\mathcal{R}_1\ne \emptyset$, then $Y$ is a one dimensional topological
   manifold. Our result improves the Handa's result \cite{Honda} that under the assumption that $1\leq \mathrm{dim}_H(Y)<2$.
\end{abstract}

\section{Introduction}
We call a length metric space, $(Y,d,p)$, a Ricci limit space if it is a Gromov-Hausdorff limit of a sequence of pointed complete Riemannian n-manifolds, $(M_i^n,p_i)$, with Ricci curvature $\mathrm{Ric}_{M_i^n}\geq -(n-1)$. In this note, we always assume that $Y$ is not a point. In papers \cite{B1,B2,B3,B4}, many important results about such limit spaces have been proven by Cheeger and Colding. By Gromov's compactness theorem, for any point $y\in Y$, given any sequence $\epsilon_j \rightarrow 0$, passing to a subsequence, $(Y,\epsilon_j^{-1}d, y)$ converges to a length space, $(T_y, 0_y)$, denoted by $(Y,\epsilon_j^{-1}d, y)\rightarrow (T_y, 0_y)$, and $(T_y, 0_y)$ is called a tangent cone at $y$ associated to $\epsilon_j\rightarrow 0.$ A tangent cone at $y$ may depend on a choice of $\epsilon_j\rightarrow 0$, even in a non-collapsing Ricci limit space $Y$ i.e., the volume of unit balls at $p_i$, $\mathrm{vol}(B_1(p_i))\geq v>0$ for all large $i$. Examples of non-collapsing Ricci limit spaces were constructed in \cite{B2},\cite{B6} whose tangent cones at a point in $Y$ are not isometric or even not homeomorphic. In a collapsing sequence, $\mathrm{vol}(B_1(p_i))\rightarrow 0$, tangent cones at a point in $Y$ may have different Hausdorff dimensions (\cite{B2}). If all tangent cones at a point, $y\in Y$, are isometric to $\mathbb{R}^k$, for some integer $k$, $y$ is called a k-regular point. The set of k-regular points in $Y$ is denoted by $\mathcal{R}_k$. A point in $Y$ is called singular if it is not a regular point. The set of singular points in $Y$ is denoted by $\mathcal{S}$. By Cheeger-Colding (\cite{B2}) and Colding-Naber (\cite{B5}), there exists an integer $k$ such that $\mathcal{R}_k$ is dense in $Y$ and has a
full renormalized measure (Lemma 2.2(2)). However, it is still open whether $\mathcal{R}_l\neq \emptyset,$ for some $l\neq k.$

The main result in this note asserts that if $\mathcal{R}_1 \neq \emptyset$, then $\bigcup_{l=2}^{n-1}\mathcal{R}_l$ is empty.
Furthermore, $Y$ is a one dimensional topological manifold with or without boundary.

\begin{Thm}
Let $Y$ be a Ricci limit space. If $\mathcal{R}_1 \neq \emptyset$ , then $Y$ is a one dimensional topological manifold
with or without boundary.
\end{Thm}

\begin{Cor} Let $Y$ be a Ricci limit space of Hausdorff-dimension $k\ge 2$. Then $\mathcal{R}_1=\emptyset$.
\end{Cor}
Let $\nu$ denote a renormalized measure (see section 1 of \cite{B2} for the definition of renormalized measure).
A geodesic in $Y$ is called a limit geodesic if it is the limit of geodesics in $M_i$.
\begin{Cor} Let $Y$ be a Ricci limit space such that the renomalized measure $\nu(Y\setminus \mathcal{R}_2)=\emptyset$. Then, for any $x, y \in Y$,
and any limit geodesic $\gamma$ connecting $x$ and $y$, interior points of $\gamma$ are all in $\mathcal{R}_2$ or
all in the singular set $\mathcal{S}$.
\end{Cor}

 Theorem 1.1 improves Theorem 1.1 in \cite{Honda}, where Honda proved that if $1\leq \mathrm{dim}_H(Y)<2$,
  then $Y$ is a one dimensional topological manifold. Note that $1\leq \mathrm{dim}_H(Y)<2$ is equivalent to that $\bigcup_{l=2}^{n-1}\mathcal{R}_l = \emptyset$ (\cite{B3}, \cite{Honda}), both imply that $\mathcal{R}_1 \neq \emptyset$. We noticed that in \cite{Honda}, Honda claimed (page 3 in \cite{Honda}) that in a paper under preparation, he can proved Theorem 1.1.

 Our proof uses the following splitting theorem of Cheeger-Colding (\cite{B1}),
 \begin{Thm}(\cite{B1})
 Let $(Y,p)$ be a Gromov-Hausdorff limit of a sequence, $\{(M_i^n, p_i)\},$ with $\mathrm{Ric}_{M_i^n}\geq -(n-1)\delta_i, \delta_i\rightarrow 0.$ If $Y$ contains a line, then $Y$ splits isometrically with a $\Bbb R^1$-facor i.e., $Y=\mathbb{R}^1\times X$ for some length metric space $X$.
 \end{Thm}

Using Theorem 1.2, we show that for $x\in \mathcal{R}_1$, there exist $\epsilon = \epsilon(x)>0, r(\epsilon)>0,$ such that if $r\leq r(\epsilon)$, $y^{\pm}\in A_{r-2\epsilon r, r+2\epsilon r}(x)=B_{r+2\epsilon r}(x)\setminus B_{r-2\epsilon r}(x),$  with $d(y^{+},x)+d(y^{-},x)-d(y^{-},y^{+})\leq 7\epsilon r,$ then $x$ lies in every geodesic connecting $y^{-}$ and $y^{+}$ (Lemma 2.1). Together with some applications of the sharp H\"older continuity on tangent cones obtained by Colding-Naber (Lemma 2.2)(\cite{B5}) and a similar argument to \cite{Honda}, theorem 1.1 can be proved.
\begin{Rem}
In \cite{B9}, Kitabeppu and Lakzian proved a similar result of Theorem 1.1 for a more general metric spaces (including Ricci limit spaces) via a complicated argument, because
the continuity of tangent cones along the interior of geodesic may not hold.
Our proof is simple but may not extend to their case. Nevertheless, Lemma 2.1 still holds in the setting of \cite{B9}.
\end{Rem}

The author would like to thank professor Shouhei Honda for his interests in this paper and for his useful suggestion that leads to Corollary 1.1.2. The author is grateful to Professors S. Honda and Tapio Rajala for bring her attention to \cite{B9}.

\section{Local structure of 1-regular points}
Let $x\in \mathcal{R}_1$. By definition, for every $\epsilon > 0$, there is $r(\epsilon)>0$ such that
for any $r\leq r(\epsilon)$,
$$d_{GH}(B_r(x),B_r^1(0))\leq \epsilon r,$$
where $B_r^1(0)\subset \mathbb{R}^1.$ Let $e_1$ be the standard basis in $\mathbb{R}^1$. Hence for every $\epsilon > 0$,
$r\leq r(\epsilon)$, there are $x_r^{+}, x_r^{-}\in B_r(x)$ such that $d(x_r^{\pm},\pm re_1)\leq \epsilon r$.
Consequently,
$$|d(x_r^{\pm},x)-r|\leq \epsilon r,$$
$$d(x_r^{+},x)+d(x_r^{-},x)-d(x_r^{+},x_r^{-})\leq 3\epsilon r.$$
\begin{Lem} Let $(Y ,d)$, $x\in \mathcal{R}_1, x_r^{+}, x_r^{-}$ be as in the above. Then there exists $\epsilon_0 > 0$
such that for any $\epsilon \leq \epsilon_0$, $r\leq r(\epsilon)$,  and any $$(y^{+},y^{-})\in
B_{\epsilon r}(x_r^{+})\times B_{\epsilon r}(x_r^{-}),$$ any geodesic $\gamma$ from $y_1$ to $y_2$ passing through $x$.
\end{Lem}
\begin{proof}
Arguing by contradiction, assume that there are $ \epsilon_i \rightarrow 0, r_i \rightarrow 0$, $(y_i^{+}, y_i^{-})\in B_{\epsilon_i r_i}(x_{r_i}^{+})\times B_{\epsilon_i r_i}(x_{r_i}^{-})$, and a geodesic, $\gamma_i$, from $y_i^{+}$ to $y_i^{-}$ such that $x\notin \gamma_i$.

By the triangle inequality,
$$|d(y_i^{+}, x)-r_i|\leq 2\epsilon_i r_i, |d(y_i^{-}, x)-r_i|\leq 2\epsilon_i r_i,$$
$$d(y_i^{+}, x)+d(y_i^{-},x) - d(y_i^{+}, y_i^{-})\leq 7\epsilon_i r_i.$$
Passing to a subsequence, we may assume that
$$(Y, r_i^{-1}d,x)\rightarrow (\mathbb{R},0),$$
and $y_i^{\pm}\rightarrow \pm 1.$

By assumption, $s_i = d(x,\gamma_i)> 0$. From the above there is $\gamma_i(t_i)$ such that
$d(x,\gamma_i)=d(x,\gamma_i(t_i))$ and
$$s_i\leq \min\{d(x, y_i^{+}),d(x, y_i^{-})\}.$$
Thus $s_i\rightarrow 0,$ as $i\rightarrow \infty$.
Without loss of generality, we can assume that
$$(Y, s_i^{-1}d, x)\rightarrow (T_x, 0_x).$$
Observe that $s_i^{-1}\gamma_i$ converges to a line in $T_x$ which has distance $1$ from $0_x$.
By Theorem 1.2 (\cite{B1}), $T_x$ is isometric to $\mathbb{R}\times W$, and because
$0_x$ is not on the line, $W$ is not a point; a contradiction to $x\in \mathcal{R}_1$.
\end{proof}
To conclude Theorem $1.1$ from Lemma 2.1, we will also need the following results of Colding-Naber (\cite{B5}).

\begin{Lem}(\cite{B5})
 If $Y$ is a Ricci limit space, then
 \begin{enumerate}
\item[(1)]
for $\nu \times \nu$ almost every pair $(a_1, a_2)\in A_1 \times A_2$, where $A_1$
and $A_2$ are subset of $Y$ that contained in a bounded ball, there exists a limit
geodesic from $a_1$ to $a_2$ whose interior lies in some $\mathcal{R}_l$,
$l$ is an integer;
  \item[(2)]
  There is an integer $k$, such that $\nu(Y\setminus \mathcal{R}_k)=0$.

   \end{enumerate}
\end{Lem}
  Let $\mathcal{WE}_k$ be the set of points, $y\in Y$, such that there exists a tangent cone at $y$ which is isometric to $ \mathbb{R}^k \times W$ and $W$ is a length metric space. Let $\underline{\mathcal{WE}}_k = \mathcal{WE}_k \setminus \mathcal{R}_k$.
\begin{Rem}$ $
\begin{enumerate}
\item[(1)]
 As an essentially direct consequence of Cheeger-Colding's work in \cite{B2,B4}, Honda (\cite{Honda}) pointed out that if $x \in \underline{\mathcal{WE}}_k$, then for all $r>0$, $\nu(B_r(x)\cap (\bigcup_{l\geq k+1}\mathcal{R}_l))>0$. Consequently, if $\nu(\cup_{l\geq k+1}\mathcal{R}_l)=0$, then $\underline{\mathcal{WE}}_k = \emptyset,$ especially, $\mathcal{R}_l = \emptyset,$ for all $l\geq k+1.$
 \item[(2)]
 By Lemma 2.2 (2), almost every pair of points in $Y$ can be connected by a limit geodesic whose
 interior is in $\mathcal{R}_k.$
\end{enumerate}
\end{Rem}
\begin{proof}[Proof of theorem 1.1]
 We complete the proof in two steps.

 First, we claim that for $x\in \mathcal{R}_1$, there exists $\tau=\tau(x)>0,$ such that $B_{\tau}(x)$ is isometric to $(-\tau,\tau).$

 Let $x\in \mathcal{R}_1, x_r^{+}, x_r^{-}, \epsilon_0, r(\epsilon)$ be chosen as in Lemma 2.1. Then for any $$(y^{+},y^{-})\in B_{\epsilon r}(x_r^{+})\times B_{\epsilon r}(x_r^{-}),$$ where $r\leq r(\epsilon)$ and $\epsilon\leq \epsilon_0\leq \frac{1}{4},$ any geodesic, $\gamma$, from $y^{+}$ to $y^{-}$, we have that $x\in \gamma$. By Lemma 2.2(1), we can assume $(y^{+},y^{-})\in B_{\epsilon r}(x_r^{+})\times B_{\epsilon r}(x_r^{-})$ and a limit geodesic $\gamma$
 from $y^{+}$ to $y^{-}$ whose interior is contained in $\mathcal{R}_l$ for some integer $l$. Since $x\in \gamma\cap \mathcal{R}_1$,
 we have that $l=1$ i.e. the interior of $\gamma$ is contained in $\mathcal{R}_1$.

 Next, we shall show that for $\tau\leq \frac{r(\epsilon_0)}{4}$, $B_\tau(x)$ is isometric to $\gamma|_{(-\tau,\tau)}$, where $\gamma$ is
 in above with a reparametrization such that $x=\gamma(0)$. If $B_\tau(x)$ is not isometric to $\gamma|_{(-\tau,\tau)}$, then there is
 a point $y \in B_{\tau}(x)\setminus \gamma$. We may assume $z\in \gamma$ such that $d(y,z)= d(y,\gamma)$. Observe that
 for every small $\delta>0$, there are $z^{\pm}\in \gamma$ such that $d(z^-,z^+)=d(z^-,z)+d(z,z^+)$ and $d(z^\pm,z)=\delta$,
 and $w$ in a geodesic from $z$ to $y$ such that $d(w,z)=\delta$ (note that $d(z^{\pm},w)\geq \delta$).
 Consequently, the tangent cone at $z$ associate to $\delta\to 0$ is isometric to $\mathbb{R}\times W$ and
 $W$ is not a point; a contradiction to $z\in \mathcal{R}_1$ (compare to the proof of Theorem 4.3, \cite{Honda}).

 A by-product of the above is that geodesics in $Y$ do not branch. By Lemma 2.2(2) and Remark 2.1(1),
 $\underline{\mathcal{WE}}_1 =\emptyset$. Consequently, any interior point $y$ in a geodesic has a
 a neighborhood isometric to $(-\tau,\tau)$ (no branching at $y$). This allow us to uniquely extend
 any geodesic in $Y$ to a maximal geodesic.

 Secondly, we will show that $Y\setminus \gamma = \emptyset$ i.e., $Y$ is a one dimensional topological manifold.
 Arguing by contradiction, assume $y\in Y\setminus \gamma$ and $z\in \gamma$ such that $d(y,z)=d(y,\gamma)$.
 If $z$ is an interior point of $\gamma$ and $\sigma$ is a geodesic from $z$ to $y$, then
 $z\in \mathcal{R}_1$ and there exists $\delta>0$, such that $B_{\delta}(z)$ is isometric to $(-\delta, \delta)$. Thus $\sigma(\delta)\in \gamma$ and $d(y,\sigma(\delta))\leq d(y,z)-\delta$, a contradiction to that $d(y,z)=d(y,\gamma)$. If $z$ is one end of $\gamma$, then
 a similar discussion yields that for any $w$ lies in the interior of $\gamma$, a geodesic from $w$ to $y$ is
 the union of the piece of $\gamma$ from $w$ to $z$ and a geodesic from $z$ to $y$ i.e., $\gamma$ extends through $z$,
 a contradiction.
\end{proof}

\begin{proof}[Proof of Corollary 1.1.2]
Let $x, y\in Y$, and $\gamma$ be a limit geodesic connecting $x$ and $y$. Let $p$ be an interior point of $\gamma$, then $p\in \mathcal{WE}_1$.
Since $\nu(Y\setminus \mathcal{R}_2)= 0$, by Remark 2.2(1) and Theorem 1.1, for any $i\neq 2$, $\mathcal{R}_i=\emptyset$. If $p \in \mathcal{S}$, there is a tangent cone, $T_p = \mathbb{R}^1\times W$, at $p$ satisfies that $\mathrm{dim}_H(W)\geq 1$ (otherwise $p \in \mathcal{R}_1$ \cite{Honda}). Thus by Proposition 3.37 in \cite{B8} and Theorem 1.1, for any renormalized measure $\nu_{\infty}$ in $T_p$, $\nu_{\infty}(T_p\setminus \mathcal{R}_2)= 0$. By splitting theorem (Theorem 1.2), using an argument as the proof of Theorem 1.1 (see also Theorem 4.3 \cite{Honda} ), we have that $T_p$ is isometric to the half plan or $\mathbb{R}^2$. Then by the sharp H\"older continuity on tangent cones obtained by Colding-Naber (\cite{B5}), we easily conclude the proof.

\end{proof}

\end{document}